\pdfoutput=1
\documentclass[11pt]{article}
\usepackage{amsmath}
\usepackage{amsthm}
\usepackage{amsfonts}
\usepackage{graphicx}
\usepackage{latexsym}
\usepackage{url}
\usepackage{bbm}
\usepackage[pdftex]{color}
\usepackage[pdftex,pdftitle={},pdfauthor={},colorlinks=TRUE,allcolors=darkblue]{hyperref}
\definecolor{darkblue}{rgb}{0,0,0.6}

\newcommand{\R}{\mathbb{R}}

\newcommand{\E}{\mathbb{E}}
\newcommand{\PP}{\mathbb{P}}
\newcommand{\N}{\mathbb{N}}
\newcommand{\Z}{\mathbb{Z}}

\newcommand{\De}{\Delta}

\newcommand{\h}{\mathcal H}
\newcommand{\f}{\mathcal F}

\newcommand{\cof}{\mathrm{COF}}

\newcommand{\toop}{\stackrel{\PP}{\longrightarrow}}
\newcommand{\schw}{\stackrel{d}{\longrightarrow}}
\newcommand{\stab}{\stackrel{st}{\longrightarrow}}

\newcommand{\ucp}{\stackrel{\mbox{\tiny u.c.p.}}{\longrightarrow}}

\newcommand{\bee}{\begin{equation}}
\newcommand{\eee}{\end{equation}}
\newcommand{\beea}{\begin{array}}
\newcommand{\eeea}{\end{array}}

\renewcommand{\theequation}{\arabic{section}.\arabic{equation}}

\theoremstyle{plain}
\newtheorem{prop}{Proposition}[section]

\newtheorem{theo}[prop]{Theorem}
\newtheorem{lem}[prop]{Lemma}

\theoremstyle{definition}

\theoremstyle{remark} 
\newtheorem{rem}[prop]{Remark}

\begin{document}

\title{Asymptotic theory for Brownian semi-stationary processes with application to turbulence}
\author{Jos\'e Manuel Corcuera\thanks{Universitat de Barcelona, Gran Via de les Corts Catalanes 585, 08007
Barcelona, Spain. Email: \href{mailto:jmcorcuera@ub.edu}{\nolinkurl{jmcorcuera@ub.edu}}.} \and Emil Hedevang\thanks{Department of Mathematics,
Aarhus University, Ny Munkegade 118, 8000 Aarhus C, Denmark. Email:  \href{mailto:emilhedevang@gmail.com}{\nolinkurl{emilhedevang@gmail.com}}.} \and Mikko S. Pakkanen\thanks{CREATES and Department
of Economics and Business, Aarhus University, Fuglesangs All\'e 4, 8210 Aarhus V, Denmark.
E-mail: \href{mailto:mpakkanen@creates.au.dk}{\nolinkurl{msp@iki.fi}}.} \and 
Mark Podolskij\thanks{Department of Mathematics, Heidelberg University,  INF 294, 69120 Heidelberg,
Germany, Email: \href{mailto:m.podolskij@uni-heidelberg.de}{\nolinkurl{m.podolskij@uni-heidelberg.de}}.}}


\maketitle

\begin{abstract}
This paper presents some asymptotic results for statistics of Brownian semi-station\-ary ($\mathcal{BSS}$) processes. More
precisely, we consider power variations of $\mathcal{BSS}$ processes, which are based on high frequency
(possibly higher order) differences of the  $\mathcal{BSS}$ model. We review the limit theory discussed in 
\cite{BCP11, BCP12} and present some new connections to fractional diffusion models. 
We apply our probabilistic results to construct a family of estimators for the smoothness parameter of the  
$\mathcal{BSS}$ process. In this context we develop estimates with gaps, which allow to obtain a valid central limit
theorem for the critical region. Finally, we apply our statistical theory to turbulence data.

\ \

\bigskip


\end{abstract}




\section{Introduction} \label{Intro}
\setcounter{equation}{0}
\renewcommand{\theequation}{\thesection.\arabic{equation}}

In this paper we study the probabilistic
limit behaviour of (realised) power variation, based on higher order
differences, in relation to the class of \emph{Brownian semi-stationary} ($%
\mathcal{BSS}$) processes. This class, which was introduced in \cite{BS09},
consists of the processes $( Y_{t}) _{t\in \mathbb{R}}$\ that
are defined by%
\begin{equation*}
Y_{t}=\mu +\int_{-\infty }^{t}g(t-s)\sigma _{s}W(\mathrm{d}s)+\int_{-\infty
}^{t}q(t-s)a_{s}\mathrm{d}s
\end{equation*}%
where $\mu $ is a constant, $W$ is a Brownian measure on $\mathbb{R}$, $g$
and $q$ are nonnegative deterministic weight functions on $\mathbb{R}$, with
$g( t) =q( t) =0$ for $t\leq 0$, and $\sigma $ and $a$
are c\`{a}dl\`{a}g processes. When $(\sigma ,a)$ is stationary and independent of $W$, then $%
Y$ is stationary, which motivates the name Brownian semi-stationary process.

In the context of stochastic modelling in turbulence the process $\sigma $
embodies the \emph{intermittency} of the dynamics. For detailed discussion
of $\mathcal{BSS}$ and the more general concept of tempo-spatial ambit
processes see \cite{BS04,BS07a,BS07b,BS08a,BS08b,BS09}. Such processes are, in particular, able to reproduce
the key stylized features of turbulence data, such as homogeneity, stationarity, skewness, and isotropy.
In general, $\mathcal{BSS}$ processes are not semimartingales (see the discussion in Section \ref{sec2-2-1} for more details). In
consequence, various important asymptotic techniques developed for semimartingales, see
for instance \cite{BGJPS,J2,PV}, such as the calculation of quadratic
variation by It\^{o} calculus algebra and those of multipower variation, do
not apply or suffice in the $\mathcal{BSS}$\ setting.

This study includes a review and some extensions of the theory developed in
\cite{BCP11,BCP12}. In this paper we will mainly consider $\mathcal{BSS}$ processes without drift (i.e. $a\equiv 0$),
which we denote by $X$. First, let us recall that a properly normalized version of {\it power variation} 
\[
\sum_{i=1}^{[t/\Delta_n]} |\Delta_i^n X|^p, \qquad \Delta_i^n X:=X_{i\Delta_n} - X_{(i-1)\Delta_n},  
\]     
constitute a consistent estimator of the stochastic quantity $\int_0^t |\sigma_s|^p ds$. Such quantities play a similar role
in turbulence as in finance: they represent a variability measure of the process $X$. The required normalization 
depends on the parameter $\alpha$, which describes the behaviour of the weight function $g$ near $0$. More precisely, we 
consider functions $g$ with $g(x)\sim x^\alpha$ as $x\downarrow 0$. We call $\alpha$ the {\it smoothness parameter} as it describes the
smoothness properties of $X$ and it corresponds to Kolmogorov's scaling law in turbulence. The parameter $\alpha$ also
controls the weak limit regime and the rate of convergence for the standardized version of power variation. For a subrange
of parameters $\alpha$ power variations turn out to be asymptotically mixed normal with convergence
rate $\Delta_n^{-1/2}$, while other $\alpha$'s lead to Rosenblatt type of limits and slower convergence rates. 
We refer to \cite{T} for the corresponding results in the Gaussian setting. 

The  main objective of the paper is to use the asymptotic theory for power variations 
to construct efficient estimators of the smoothness parameter appearing in the weight function $g$.
It turns out that it is preferable to consider higher order increments when computing power variation statistics.
In our setting using higher order differences has the following crucial advantages: (a) The asymptotic mixed normality of power 
variations becomes valid for all smoothness parameters of interest (a known effect for Gaussian models; see e.g. \cite{IL}),
(b) The power variations are more robust to  the presence of smooth drift functions. We will review these properties 
in Section \ref{limit} and present some new results. In the next step, we will apply the asymptotic theory for power variations
to construct a consistent estimator of the parameter $\alpha$. Our main tools are realised variation ratios, which compare
power variations at two different sampling frequencies ($\Delta_n$ and $2\Delta_n$). In the critical regime we will use
estimators with gaps to obtain rate optimal estimators of $\alpha$. This type of estimators were studied by Lang and Roueff \cite{LR}
in the purely Gaussian framework. 

Finally, we will apply our statistical methods to a turbulence data set obtained from one-point measurements of the
longitudinal component of the wind velocity in the atmospheric boundary layer. Our estimator lies around $\alpha =-1/6$,
a parameter value that corresponds to the celebrated Kolmogorov's $5/3$-law. Furthermore, our statistics show
a rather stable behaviour when the power $p$ is ranging between $1.5$ and $2.5$. These facts demonstrate that $\mathcal{BSS}$
processes constitute adequate models for turbulent flows. 

The paper is organized as follows. In Section \ref{definition} we introduce common
notation, definitions, main assumptions and some probabilistic properties of
the procesess under consideration; in particular, we establish new connections between the $\mathcal{%
BSS}$ processes and certain fractional diffusions processes. In Section \ref{limit} we develop the limit
results, in Section \ref{statistics} we construct estimators for the smoothness parameter
of the weight function $g$, study their properties and in Section \ref{turbulence} we apply them to
turbulence data.

\section{Definitions and first probabilistic properties}\label{definition}
\setcounter{equation}{0}
Let $(\Omega ,\mathcal{F},(%
\mathcal{F}_{t})_{t\in \mathbb{R}},\mathbb{P})$ be a given filtered
probability space. We shall consider a $\mathcal{BSS}$ process $(X_{t})_{t\in \mathbb{%
R}}$, without drift, defined by
\begin{equation}
X_{t}=\int_{-\infty }^{t}g(t-s)\sigma _{s}W(ds),\qquad t\in \mathbb{R},
\label{x}
\end{equation}%
where $W$ is an $(\mathcal{F}_{t})_{t\in \mathbb{R}}$-adapted white noise on $\mathbb{R}$, $g:\mathbb{R}%
^{+}\rightarrow \mathbb{R}$ is a deterministic weight function satisfying $%
g\in L^{2}(\mathbb{R}^{+})$, and $\sigma $ is an \ $(\mathcal{F}_{t})_{t\in \mathbb{R}}$%
-adapted c\`{a}dl\`{a}g intermittency process. By an $\left( \mathcal{F}%
_{t}\right) $-adapted white noise we understand a zero-mean Gaussian random measure on Borel sets $A\subset \mathbb{R}$ such that $m(A) <\infty$, with covariance
\begin{equation*}
\E[W(A)W(B)] = m(A\cap B),
\end{equation*}%
where $m$ is the Lebesgue measure. Moreover, if $A\subset [ t,\infty )$
then $W(A)$ is independent of $\mathcal{F}_{t}$ and if $A\subset (-\infty,t]$ then $W(A)$ is $\mathcal{F}_{t}$-measurable. Such a random measure is a particular case of
a \textit{martingale measure.}  Note that for $a\in \mathbb{R}$ the process $\big(W([a,t])\big)_{t\geq a}$ is a standard
Brownian motion. Integrals with respect to martingale measures are an
extension of It\^{o} integrals, see \cite{W} for their definition.

In order to guarantee the a.s.\ finiteness 
of the integral, we furthermore assume that 
\begin{align} \label{finite}
\int_{-\infty}^t g^2(t-s) \sigma_s^2 ds <\infty \quad \mbox{a.s.}, 
\end{align}
for any $t\in \R$. This condition is obviously satisfied when $(\sigma_t)_{t\in \R}$ is a stationary process 
with a finite second moment (recall that $g\in L^2(\R^{+})$). We assume that the underlying observations 
are
\[
X_{i\Delta_n}, \qquad i=0,\ldots, [t/\Delta_n],
\]    
with $t>0$ being fixed and $\Delta_n\rightarrow 0$. This type of sampling is customarily called infill asymptotics
or high frequency data. 
Our class of statistics is based upon higher order differences of the $\mathcal{BSS}$ process $X$. We briefly recall
the definition: For any $k\in \N$, the $k$-th order difference $\Delta_{i,k}^{n,v} X$ at frequency $v\Delta_n$, where $v\in \N$, and at stage $i\geq vk$ is defined 
by
\[
\Delta_{i,k}^{n,v} X:= \sum_{j=0}^k (-1)^j \binom{k}{j} X_{(i-vj)\Delta_n}.
\]
However, when $v = 1$ we usually write $\Delta_{i,k}^{n}X$ instead of $\Delta_{i,k}^{n,1}X$.
For example,
\[
\Delta_{i,1}^n X = X_{i\Delta_n}-X_{(i-1)\Delta_n} \quad \textrm{and} \quad \Delta_{i,2}^n X = X_{i\Delta_n}-2X_{(i-1)\Delta_n}+X_{(i-2)\Delta_n}.
\]
In this paper we will restrict our attention to the higher order differences with $k\geq 1$, although other filters 
may be considered. In order to define power variations of the process $X$ we need to introduce some further 
notation. We consider a centered stationary Gaussian process $G=(G_t)_{t\in \R}$, which we will call
the Gaussian core of $X$, that is given as
\begin{align} \label{g}
G_t: = \int_{-\infty}^t g(t-s)  W(ds), \qquad t\in \R.
\end{align}  
Note that $G_t<\infty$ since $g\in L^2(\R^{+})$. The correlation kernel $r$ of $G$ is given via 
\[
r(t) = \frac{\int_0^\infty g(u)g(u+t) du}{\|g \|_{L^2(\R^{+})}^2}, \qquad t\geq 0.
\]
Most crucial quantity for the asymptotic theory is the variogram  $R$, i.e.
\begin{align} \label{R}
R(t):= \E[(G_{t+s} - G_s)^2] = 2 \|g \|_{L^2(\R^{+})}^2 (1-r(t)).
\end{align} 
We introduce two closely related  power variations based on increments $\Delta_{i,k}^{n,v}$ of order $k$ 
at frequency $v\Delta_n$ as 
\begin{align} 
\label{powervarpur}
V(X,p,k,v;\Delta_n)_t &:= \sum_{i=vk}^{[t/\Delta_n]} |\Delta_{i,k}^{n,v} X|^p, \\
\label{powervar}
\bar{V}(X,p,k,v;\Delta_n)_t &:= \Delta_n \tau_{k}(v\Delta_n)^{-p}V(X,p,k,v;\Delta_n)_t, 
\end{align} 
where $p$ is a positive power and $\tau_{k}(v\Delta_n):=\sqrt{\E[|\Delta_{i,k}^{n,v} G|^2]}$. 
To determine the limiting 
behaviour of the statistics $\bar{V}(X,p,k,v;\Delta_n)_t$, we need to introduce a set of assumptions, which we discuss
in the next subsection.

\subsection{Main assumptions}
We start with various conditions on the weight function $g: \R^{+} \rightarrow \R$. Below, all functions $L_f:\R^{+} \rightarrow \R$,
indexed by a given mapping $f$, are continuous and slowly varying at $0$, i.e. 
$\lim_{x\downarrow  0} L_f(tx)/L_f(x) =1$ for any $t>0$. Furthermore, the function $f^{(m)}$ denotes the $m$-th derivative
of $f$ and $\alpha$ denotes a number in $(-\frac 12, 0)\cup (0, \frac 12)$.  \\ \\
\hypertarget{A1}{(A1)}: It holds that \\ \\
(i) $g(x)= x^\alpha  L_g(x)$. \\ \\
(ii) $g^{(k)}(x)= x^{\alpha -k} L_{g^{(k)}} (x)$ and, for any $\varepsilon >0$, we have 
$g^{(k)}\in L^2((\varepsilon, \infty ))$. Furthermore, $|g^{(k)}|$ is non-increasing 
on the interval $(a,\infty)$ for some $a>0$. \\ \\
(iii) For any $t>0$
\begin{align} \label{F}
F_t= \int_1^\infty |g^{(k)}(s)|^2 \sigma_{t-s}^2 ds <\infty.  
\end{align} 
We remark that in the case of $\alpha =0$ the assumptions (\hyperlink{A1}{A1})(i) and (\hyperlink{A1}{A1})(ii) are never satisfied simultaneously,
hence we excluded this case. The next set of assumptions deals with the variogram $R$. \\ \\
\hypertarget{A2}{(A2)}: For the smoothness parameter $\alpha$ from (A1) it holds that \\ \\
(i) $ R(x) = x^{2\alpha+1} L_{R}(x)$. \\ \\
(ii) $ R^{(2k)}(x) = x^{2\alpha -2k+1} L_{R^{(2k)}}(x)$. \\ \\   
(iii) There exists a $b\in (0,1)$ such that 
\begin{equation*} 
\limsup_{x\downarrow 0} \sup_{y\in [x,x^b]} \Big|\frac{L_{R^{(2k)}}(y)}{L_{R}(x)} \Big| <\infty . 
\end{equation*} 
This set of assumptions is standard in the literature; see e.g. \cite{BCP12, GL}. These conditions are required
to develop the asymptotic theory for power variation of higher order differences of the Gaussian process $G$
defined at \eqref{g}. 
Although the variogram $R$ is uniquely determined by the weight function $g$, assumption (\hyperlink{A1}{A1}) does not
imply (\hyperlink{A2}{A2}) in general. However, when the slowly varying function $L_g$ is smooth enough and satisfies  
$\lim_{x\downarrow 0}L_g(x) = c\in (0,\infty)$, the condition (\hyperlink{A1}{A1})(i) would naturally imply all other conditions
from (\hyperlink{A1}{A1}) and (\hyperlink{A2}{A2}) except (\hyperlink{A1}{A1})(iii). Finally, we present an assumption on the smoothness of the process 
$\sigma$, which is required for the proof of the central limit theorems. \\ \\
\hypertarget{A3}{(A3-$\gamma$)}: For any $q>0$, it holds that 
\begin{align} \label{holder}
\E[|\sigma_t -\sigma_s|^q] \leq C_q|t-s|^{\gamma q}  
\end{align}
for some $\gamma >0$ and $C_q>0$. In particular, condition \eqref{holder} implies by the Kolmogorov's
criteria  that the process $\sigma$
has $\beta$-H\"older continuous paths for all $\beta\in (0,\gamma )$. \\ \\
The assumptions (\hyperlink{A1}{A1}) and (\hyperlink{A2}{A2}) imply certain probabilistic properties of the $\mathcal{BSS}$ process $X$, which 
we explain in the next subsection.

\subsection{Some probabilistic properties} 
This subsection is devoted to probabilistic properties of the processes $G$ and $X$, which are direct consequences 
of assumptions (\hyperlink{A1}{A1}) and (\hyperlink{A2}{A2}).

\subsubsection{Is a $\mathcal{BSS}$ process a semimartingale?} \label{sec2-2-1}
After defining the class of $\mathcal{BSS}$ processes one naturally asks if this is a subclass of continuous semimartingales.
We start by exploring this question for the Gaussian core $G$ defined at \eqref{g}, as it directly influences
the fine structure of the process $X$. Observing the decomposition
\[
G_{t+\Delta } - G_t = \int_t^{t+\Delta} g(t+\Delta -s) W(ds) + \int_{-\infty}^t \{g(t+\Delta -s) - g(t-s) \} W(ds), 
\]
we obtain by formal differentiation 
\[
dG_t = g(0+) dW(t) + \Big(\int_{-\infty}^t g'(t-s) W(ds)\Big) dt,
\]
where we use the convention $g'=g^{(1)}$. Indeed, the Gaussian process $G$ is an It\^o semimartingale 
when $g(0+)<\infty$ and $g'\in L^2(\R^+)$ and this property also transfers to the $\mathcal {BSS}$ process $X$
under mild assumptions. However, this situation is less interesting from the theoretical point of view,
since the asymptotic behaviour of power variation of continuous semimartingales is rather well-understood; we refer to \cite{BGJPS,J2,PV} for more details. The
assumption (\hyperlink{A1}{A1})(i) implies that
\begin{equation*}
g\in L^{2}(\mathbb{R}^{+})\qquad \mbox{but}\qquad g^{\prime }\not\in L^{2}(%
\mathbb{R}^{+}),
\end{equation*}%
because the derivative of $g(x)=x^{\alpha }L_{g}(x)$ is not square
integrable near $0$ for $\alpha \in (-\frac{1}{2},0)\cup (0,\frac{1}{2})$.
It can be shown, see \cite{B08}, that the conditions $g(0+)<\infty $ and   $%
g^{\prime }\in L^{2}(\mathbb{R}^{+})$ are also necessary conditions for $%
G$ to be a semimartingale. Hence, the process $G$, and so the process $X$
(unless $\sigma =0$), is not a semimartingale.

\subsubsection{Asymptotic correlation structure and some consequences} \label{asycorr}
Now, we turn our attention to the local behaviour of the Gaussian core $G$. Assumption (\hyperlink{A2}{A2})(i) suggests that 
the small scale behaviour of the increments of $G$ is similar to the small scale behaviour of the increments of $B^H$,
where $B^H$ is a fractional Brownian motion with Hurst parameter $H=\alpha + 1/2\in (0,1)$. This connection can 
be formalized as follows: Let $r_{k,n}$ be the correlation structure of higher order increments $\Delta_{i,k}^n G$, i.e.
\begin{align} \label{rn}
r_{k,n} (j):= \mbox{corr} (\Delta_{1,k}^n G, \Delta_{1+j,k}^n G), \qquad j\geq 0. 
\end{align}  
Then the polarization argument and condition (\hyperlink{A2}{A2})(i) imply the convergence
\begin{align} \label{rnconv}
\lim_{n\rightarrow \infty } r_{k,n} (j) = \rho_k (j),  
\end{align}         
where $\rho_k $ is the correlation function of the $k$-th order increments $\Delta_{i,k} B^H$ that are defined by
\[
\Delta_{i,k} B^H:= \sum_{j=0}^k (-1)^j \binom{k}{i} B^H_{i-j} 
\qquad i\geq k, \quad k\geq 2,
\] 
and $H=\alpha + 1/2$. For instance, for $k=1$ it obviously holds by assumption (\hyperlink{A2}{A2})(i) that
\[
r_{1,n} (j) = \frac{R((j+1)\Delta_n) - 2 R(j\Delta_n) + R((j-1)\Delta_n)}{2R(\Delta_n)} \rightarrow 
\frac{1}{2} \Big(|j+1|^{2H} - 2|j|^{2H} + |j-1|^{2H} \Big),
\]
where the right side is the correlation function of the fractional Brownian noise. Consequently, the limit
theory (law of large numbers and the central limit theorem) for the power variation $\bar{V}(G,p,k;\Delta_n)_t$ 
is expected to be the same as for $\bar{V}(B^H,p,k;\Delta_n)_t$ with $H=\alpha + 1/2$. The latter is well-understood 
since the work of \cite{BM}. We shortly recall this classical result. Let 
\begin{align} \label{f}
f(x)&=|x|^p - m_p, \\
\label{mp}
m_p&=\E[|U|^p], \qquad U\sim \mathcal N(0,1). 
\end{align}
Then the function $f$, which is associated with a centered version of the statistic 
\begin{equation*}
V(B^H,p,k,1;\Delta_n)_t,
\end{equation*}
exhibits a Hermite expansion of the form
\begin{align} \label{fherm}
f(x)=\sum_{l=2}^{\infty} a_l H_l(x),
\end{align} 
where $(H_l)_{l\geq 0}$ are Hermite polynomials and $a_2 \not = 0$. The number $2$, which is the minimal index $l$
with $a_l \not =0$, is called the Hermite rank of the function $f$. Now, the statistic $\bar{V}(B^H,p,k;\Delta_n)_t$ is asymptotically
normal (with a standard $\Delta_n^{-1/2}$ rate) whenever the condition
\[
\sum_{j=1}^{\infty} \rho_k (j)^2<\infty 
\]  
holds, where the power reflects the Hermite rank of $f$. A simple computation shows that $|\rho_k (j)|= O( j^{2(H-k)})$ as
$j\rightarrow \infty$. Hence, the preceding condition is satisfied for (a) $k=1$ and $H\in (0, 3/4)$, (b) $k\geq 2$ 
and $H\in (0,1)$. When $k=1$ we obtain the following cases:
\begin{align*}
0<H<3/4 &: \qquad \Delta_n^{-1/2}\left( \bar{V}(B^H,p,1,1;\Delta_n)_t - m_p t \right) \schw v_p W'_t, \\
H=3/4 &:   \qquad (\Delta_n \log{\Delta_n^{-1}})^{-1/2}\left( \bar{V}(B^H,p,1,1;\Delta_n)_t - m_p t \right) \schw \tilde{v}_p W'_t, \\
3/4<H<1 &:   \qquad \Delta_n^{2H-2}\left( \bar{V}(B^H,p,1,1;\Delta_n)_t - m_p t \right) \schw L_t,
\end{align*}
where the weak convergence takes place on $\mathbb D([0,T])$ equipped with the uniform topology, $W'$ denotes a Brownian motion
and $L$ is a Rosenblatt process (see e.g. \cite{T}). Finally, the constants $v_p$ and $\tilde{v}_p$ are given by
\begin{align*}
v_p &:= \sum_{l=2}^\infty l! a_l^2 \Big(1+ 2 \sum_{j=1}^\infty \rho_1(j)^l \Big), \\
\tilde{v}_p  &:= 2 \lim_{n\rightarrow \infty} \frac{1}{\log n} \sum_{j=1}^{n-1} \frac{n-k}{n} \rho_1(j)^2 \cdot
\sum_{l=2}^\infty l! a_l^2.
\end{align*}

\subsubsection{Connection to integral processes} \label{integralpr}
In the papers \cite{BCP09,CNW,NNT} the asymptotic theory for power variation of integrals with respect to Gaussian 
processes and related functionals have been developed. A brief comparison of the asymptotic results for such integral 
processes and $\mathcal{BSS}$ models shows that the limit theory is quite similar. Thus, a natural question appears: 
Do $\mathcal{BSS}$ models exhibit a representation as an integral with respect to a Gaussian process? First, we observe 
that
\[
X_t \not = \int_0^t \sigma_s dG_s,
\] 
where $G$ is a Gaussian core of $X$. However, these two processes are indeed related. Let us assume for the moment
that the process $\sigma$ is deterministic. We introduce the decomposition $X_t= X_t^{\prime} + X_t^{\prime \prime}$
with 
\begin{align*}
X_t^{\prime} &= \int_0^t g(t-s) \sigma_s dW_s, \\
X_t^{\prime \prime} &= \int_{-\infty}^0 g(t-s) \sigma_s dW_s.
\end{align*}
Similarly, $G_t= G_t^{\prime} + G_t^{\prime \prime}$, where $G^{\prime},G^{\prime \prime}$ are defined 
exactly as $X^{\prime},X^{\prime \prime}$ with $\sigma =1$. Applying the integration theory developed in  
\cite{MV}, we deduce the identity
\[
\int_0^t \sigma_s dG_s^{\prime} = X_t^{\prime} + \int_0^t \left( \int_0^s (\sigma_s - \sigma_u) g'(s-u) W(du) \right) ds.
\]
We remark that the integral on the left side is well-defined
in the Young sense when the assumption (\hyperlink{A3}{A3-$\gamma$}) is satisfied
for some $\gamma \in (1/2-\alpha ,1)$, since the process $G^{\prime}$
has H\"older continuous paths of all orders smaller than $\alpha +1/2$. A straightforward computation shows that the integral
on the right side is well-defined in the It\^o sense under the same condition (\hyperlink{A3}{A3-$\gamma$}) with $\gamma \in (1/2-\alpha ,1)$.
   
The second part $X^{\prime \prime}$ satisfies the differential equation 
\[
d X_t^{\prime \prime} = \left( \int_{-\infty}^0 g'(t-s) \sigma_s W(ds) \right) dt
\]
and similar identity holds for $G^{\prime \prime}$. Notice that the involved Brownian integral is finite due to assumption
(\hyperlink{A1}{A1})(iii) applied to $k=1$. Summarizing these findings we conclude that $X_t$ can be decomposed as
\[
X_t = \int_0^t \sigma_s dG_s + A_t,
\]  
where $A$ is a continuously differentiable process. Thus, the law of large numbers for power variations of $X$ and 
$\int \sigma_s dG_s$ coincide, since the process $A$ is too smooth to affect the limit. However, the central limit 
theorem for the power variation of $X$ may be seriously affected by the presence of the drift $A$. We will study this type
of effects in the next section.

\section{Limit theory for power variations} \label{limit}
\setcounter{equation}{0}
Before we present the main limit theorems for power variation of the $\mathcal {BSS}$ process $X$, let
us introduce some further definitions. 
For elements $Y^n, Y$ from the space of c\'adl\'ag function $\mathbb D([0,T])$, we write $Y^n \ucp Y$
whenever $\sup_{[0,T]} |Y^n_t - Y_t| \toop 0$. We say that a sequence of processes $Y^n$ converges stably in law to
a  process $Y$, where $Y$ is defined on an extension $(\Omega', \mathcal{F}', \mathbb P')$ 
of the original probability $(\Omega, \mathcal{F}, \mathbb P)$, in the space $\mathbb D([0,T])$
equipped with the uniform topology ($Y^n \stab Y$) if and only if
\begin{equation*} 
\lim_{n\rightarrow \infty} \E[f(Y^n) Z] = \E'[f(Y)Z] 
\end{equation*} 
for any bounded and continuous function $f: \mathbb D([0,T]) \rightarrow \R$ and any bounded $\mathcal F$-meas\-urable
random variable $Z$.  We refer to \cite{AE}, \cite{JS} or \cite{R} for a detailed study of stable convergence. 
Note that stable convergence is a stronger mode of convergence than weak convergence, but it is weaker that 
u.c.p. convergence.

First, we recall the law of large numbers for the statistic $\bar{V}(X,p,k,v;\Delta_n)_t$. The corresponding result
for $k=1$ was proved in \cite{BCP11}, while the case $k=2$ was treated in \cite{BCP12}. The extension to a general
$k\geq 1$ is straightforward, and therefore omitted.

\begin{theo} \label{th1}
Assume that the conditions (\hyperlink{A1}{A1}) and (\hyperlink{A2}{A2}) hold. Then we obtain that 
\begin{align} \label{lln}
\bar{V}(X,p,k,v;\Delta_n)_t \ucp V(X,p)_t:= m_p \int_0^t |\sigma_s|^p ds,
\end{align}
where the power variation $\bar{V}(X,p,k,v;\Delta_n)_t$ is defined at \eqref{powervar} and the constant $m_p$ is given
by \eqref{mp}.
\end{theo} 

We remark that the statistic $m_p^{-1}\bar{V}(X,p,k,v;\Delta_n)_t$ is a consistent estimator of the stochastic quantity
$\int_0^t |\sigma_s|^p ds$. However, the computation of this statistic requires the knowledge of the 
parameter $\tau_{k}(v\Delta_n)$, which in turn depends on the weight function $g$. Nevertheless, even when $g$ is unknown,
Theorem \ref{th1} may be applied to estimate the smoothness parameter $\alpha$. We will discuss the estimation procedure
in the following section.

Now, we present the central limit theorem associated with the u.c.p. convergence at \eqref{lln}. 
More precisely, we state the central limit theorem for the properly standardized vector
$(\bar{V}(X,p,k,1;\Delta_n), \bar{V}(X,p,k,2;\Delta_n))$, since it will be required in the next section.  The proof for $k=1$ (resp. $k=2$)
can be found in \cite{BCP11} (resp. \cite{BCP12}). The extension to the general case of $k\geq 1$ follows along
the lines of the proof in \cite{BCP12}.  

\begin{theo} \label{th2}
Assume that the conditions (\hyperlink{A1}{A1}), (\hyperlink{A2}{A2}) hold and (\hyperlink{A3}{A3-$\gamma$}) is satisfied for some $\gamma \in (0,1)$
with $\gamma (p\wedge 1) > 1/2$. If $k=1$ we further assume that $\alpha \in (-\frac 12,0)$. Then we obtain the stable convergence
\begin{align} \label{clt}
\Delta_n^{-1/2} \Big(\bar{V}(X,p,k,1;\Delta_n)_t - V(X,p)_t,
\bar{V}(X,p,k,2;\Delta_n)_t - V(X,p)_t  \Big) \stab  \int_0^t |\sigma_s|^p \Lambda_p ~dB_s
\end{align}
on $\mathbb D^2([0,T])$ equipped with the uniform topology, where 
$B$ is a $2$-dimensional Brownian motion that is defined on an
extension of the original probability space $(\Omega, \mathcal{F}, \mathbb P)$ and is independent of $\mathcal{F}$, 
and the matrix $\Lambda_p=(\lambda_p^{ij})_{1\leq i,j\leq 2}$ is  given by
\begin{align} 
\lambda_p^{11} &= \lim_{n\rightarrow \infty } \De_n^{-1} \mathrm{var} \Big( \bar{V}(B^H,p,k,1;\Delta_n)_1\Big), \qquad 
\lambda_p^{22} = \lim_{n\rightarrow \infty } \De_n^{-1} \mathrm{var} \Big( \bar{V}(B^H,p,k,2;\Delta_n)_1\Big) \nonumber \\
\label{lambda} 
\lambda_p^{12} &= \lim_{n\rightarrow \infty } \De_n^{-1} \mathrm{cov} \Big( \bar{V}(B^H,p,k,1;\Delta_n)_1, 
\bar{V}(B^H,p,k,2;\Delta_n)_1 \Big). 
\end{align}
with $B^H$ being a fractional Brownian motion with Hurst parameter $H=\alpha + 1/2$.
\end{theo} 

\begin{rem} \rm \label{rem1}
We remark that the constants $\lambda_p^{ij}$ defined at \eqref{lambda} are indeed finite. For instance, it holds that 
\[
\lambda_p^{11} = \sum_{l=2}^\infty l! a_l^2 \Big(1+ 2 \sum_{j=1}^\infty \rho_k(j)^l \Big),
\]
where $(a_l)_{l\geq 2}$ are Hermite coefficients that appear in the Hermite expansion in \eqref{fherm} and 
the correlation function $\rho_k(j)$ is defined at \eqref{rnconv}. The finiteness of $\lambda_p^{11}$ follows directly
from $|\rho_k(j)|\leq 1$, $\sum_{j=1}^\infty \rho_k(j)^2<\infty$ (see section \ref{asycorr}) and the identity
\begin{equation}\label{variance}
\infty > \mbox{var}[|U|^{p}] = m_{2p}-m^2_p = \sum_{l=2}^\infty l! a_l^2, \qquad U\sim \mathcal N(0,1).
\end{equation}
\qed
\end{rem}

\begin{rem} \label{rem2}
There are various multivariate extensions of Theorem \ref{th2}. We refer to \cite{BCP11,BCP12} for joint
stable central limit theorems for the family 
\[
\Delta_n^{-1/2} \left( V(X,p_j,k,v_j;\Delta_n)_t - V(X,p_j)_t \right)_{1\leq j\leq d},
\]
where $(p_j)_{1\leq j\leq d}$ are positive powers and $(v_j)_{1\leq j\leq d}$ are natural numbers. Such joint
central limit theorems are important, because our estimator of the smoothness parameter $\alpha$ is a ratio 
of power variation statistics compared at two different frequencies ($\Delta_n$ and $2\Delta_n$). \qed  
\end{rem}
We remark that taking higher order difference leads to a central limit theorem for all values of $\alpha$
in the interval $(-\frac 12, 0)\cup (0, \frac 12)$, while the convergence at \eqref{clt} for $k=1$ holds 
only when $\alpha \in (-\frac 12, 0)$. This effect of higher order filters is well-known for Gaussian processes;
see e.g. \cite{IL}. Another advantage of higher order differences is its robustness to smooth distortions. Let 
$C^{m}([0,T])$ denote the class of functions $f$ that are $[m]$ times continuously differentiable with
$f^{([m])}$ being H\"older continuous of order $m-[m]$. We obtain the following lemma.

\begin{lem} \label{lem1}
Let $(A_t)_{t\in [0,T]}$ be a stochastic process with paths in $C^{m}([0,T])$ for some $m\geq 1$. Consider 
the process
\[
Z_t= X_t + A_t.
\]
Under the assumptions (\hyperlink{A1}{A1}) and (\hyperlink{A2}{A2}) it holds that
\begin{align} \label{robust}
|\bar{V}(Z,p,k,1;\Delta_n)_t - \bar{V}(X,p,k,1;\Delta_n)_t| 
= O_{\mathbb P}(|\Delta_n^{k\wedge m -\alpha -1/2} L(\Delta_n)|^{p\wedge 1}),
\end{align}
where $L: \R^+ \rightarrow \R$ is a slowly varying function. In particular, the rate on the right hand side converges to $0$.  
\end{lem} 
\begin{proof}
The condition $A\in C^{m}([0,T])$ implies that 
\[
|\Delta_{i,k}^n A|\leq C \Delta_n^{k\wedge m}
\]
for some constant $C>0$. On the other hand we have that 
\[
\tau_{k}(\Delta_n) = \Delta_n^{\alpha + 1/2} L(\Delta_n)
\]
for some slowly varying function $L$, due to assumption (\hyperlink{A2}{A2})(i). Assume that $p\in (0,1]$. Then we deduce 
\begin{multline*} 
|\bar{V}(Z,p,k,1;\Delta_n)_t - \bar{V}(X,p,k,1;\Delta_n)_t| \\ \leq 
\Delta_n \tau_{k}(\Delta_n)^{-p}\sum_{i=k}^{[t/\Delta_n]} |\Delta_{i,k}^n A|^p
 = O_{\mathbb P}(|\Delta_n^{k\wedge m}\tau_{k}(\Delta_n)^{-1}|^{p}),
\end{multline*} 
which implies the assertion of Lemma \ref{lem1} for $p\in (0,1]$. For $p>1$ we apply the mean value theorem 
to conclude that 
\begin{multline*}
|\bar{V}(Z,p,k,1;\Delta_n)_t - \bar{V}(X,p,k,1;\Delta_n)_t| \\ \leq 
p\Delta_n \tau_{k}(\Delta_n)^{-p}\sum_{i=k}^{[t/\Delta_n]} 
(|\Delta_{i,k}^n A|+ |\Delta_{i,k}^n X|)^{p-1} |\Delta_{i,k}^n A|.
\end{multline*}
Using that   
$$\Delta_n \tau_{k}(\Delta_n)^{-(p-1)}\sum_{i=k}^{[t/\Delta_n]} |\Delta_{i,k}^n X|^{p-1} \ucp m_{p-1} \int_0^t |\sigma_s|^{p-1} ds$$
and $|\Delta_{i,k}^n A|\leq C \Delta_n^{k\wedge m}$ with $m\geq 1$, we deduce the approximation  
\[
|\bar{V}(Z,p,k,1;\Delta_n)_t - \bar{V}(X,p,k,1;\Delta_n)_t| 
= O_{\mathbb P}(\Delta_n^{k\wedge m}\tau_{k}(\Delta_n)^{-1}),
\]  
which implies the assertion of Lemma \ref{lem1} for $p>1$.
\end{proof} 
We remark that the degree of robustness to drift processes $A\in C^{m}([0,T])$ is increasing in $k$. In fact,  
for $k=[m]+1$ the right side of \eqref{robust} converges to $0$ at the fastest rate. Since $m\geq 1$, the statistics 
$\bar{V}(Z,p,k,1;\Delta_n)_t$ and $\bar{V}(X,p,k,1;\Delta_n)_t$ are asymptotically equivalent. They also satisfy
the same central limit theorem given $(k\wedge m -\alpha -1/2)(p\wedge 1)>1/2$. For instance, when $k=1$ and $p\geq 1$, 
the quantity $\bar{V}(Z,p,k,1;\Delta_n)_t$ satisfies Theorem \ref{th2} only for $\alpha \in (-1/2,0)$.

\begin{rem} \rm
As we mentioned in section \ref{integralpr}, when $\sigma$ is deterministic and (\hyperlink{A3}{A3-$\gamma$}) is satisfied for
some $\gamma \in (1/2-\alpha ,1)$ (this means that $\sigma$ is H\"older continuous of order $\gamma \in (1/2-\alpha ,1)$),
we have the decomposition 
\[
X_t = \int_0^t \sigma_s dG_s + A_t =: \widetilde{X}_t+A_t
\] 
with $A\in C^{1}([0,T])$. For $k=1$ the central limit theorem for the statistics 
$\bar{V}(\widetilde{X},p,k,1;\Delta_n)_t$ holds for $\alpha \in (-\frac{1}{2},0) \cup (0, \frac 14)$
(see e.g. \cite{BCP09}), which is in line with the discussion
in section \ref{asycorr}. However, when $k=1$, Theorem \ref{th2} holds for the quantity
$\bar{V}(X,p,k,1;\Delta_n)_t$ only if $\alpha \in (-\frac{1}{2},0)$. This follows from the result of Lemma \ref{lem1},
although the rate in \eqref{robust} is by no means sharp. \qed    
\end{rem}

\section{Estimation of the smoothness parameter} \label{statistics}
\setcounter{equation}{0}

We apply the asymptotic theory presented in the previous section to construct consistent and asymptotically normal estimators of the smoothness parameter $\alpha$. 
We consider mainly the specification
\begin{equation}\label{gammakernel}
g(x) = x^\alpha \exp(-\lambda x),\qquad x>0,
\end{equation}
where $\lambda>0$ is another parameter, but some of the results apply also to a general $g$ that merely satisfies conditions (\hyperlink{A1}{A1}) and (\hyperlink{A2}{A2}). 
We use the change-of-frequency statistic based on second order increments
\[
\cof(p,\Delta_n)_t = \frac{V(X,p,2, 2;\Delta_n)_t}{V(X,p,2,1;\Delta_n)_t},\qquad t>0,
\]
with $p>0$, as a foundation for our first estimator. To understand the asymptotic behavior of $\cof(p,\Delta_n)_t$ we write
\[
\cof(p,\Delta_n)_t = \bigg(\frac{\tau_2(2\Delta_n)^2}{\tau_2(\Delta_n)^2}\bigg)^{p/2}\frac{\bar{V}(X,p,2, 2;\Delta_n)_t}{\bar{V}(X,p,2,1;\Delta_n)_t}.
\]
Under Assumption (\hyperlink{A2}{A2}), we have
\[
\frac{\tau_2(2 \Delta_n)^2}{\tau_2(\Delta_n)^2} = \frac{4 R(2\Delta_n)-R(4\Delta_n)}{4 R(\Delta_n)-R(2\Delta_n)} \rightarrow 2^{2 \alpha +1},
\]
so by Theorem \ref{th1}, the following consistency result is immediate.
\begin{prop} \label{propest} Under the conditions of Theorem \ref{th1} we have for any $p>0$,
\[
\cof(p,\Delta_n)_t \ucp 2^{\frac{(2\alpha +1)p}{2}}
\] 
and, consequently,
\begin{equation}\label{cofconsistency}
\hat{\alpha}(p,\Delta_n,t) := h_p\big(\cof(p,\Delta_n)_t\big) \ucp \alpha,
\end{equation}
where
\[
h_p(x) = \frac{\log_2(x)}{p}- \frac{1}{2}, \qquad x>0,
\]
with $\log_2$ standing for the base-$2$ logarithm.
\end{prop}
Note that, since \eqref{cofconsistency} holds for any $p>0$, computing the change-of-frequency statistics for different values of $p$ can be used to gauge the robustness of the estimates of $\alpha$.

Next we derive a feasible, asymptotically normal test statistic in the case $p \geq 2$ (see, however, Remark \ref{plessthantwo} below). To this end, let us consider the decomposition
\begin{multline}\label{cofdecomp}
\Delta^{-1/2}_n \big(\cof(p,\Delta_n)_t - 2^{\frac{(2\alpha +1)p}{2}}\big) \\
\begin{aligned}
& = \bigg(\frac{\tau_2(2\Delta_n)^2}{\tau_2(\Delta_n)^2}\bigg)^{p/2} \bigg( \frac{\Delta^{-1/2}_n\big(\bar{V}(X,p,2,2;\Delta_n)_t-V(X,p)_t\big)}{V(X,p)_t} \\
& \qquad - \frac{\bar{V}(X,p,2,2;\Delta_n)_t}{\bar{V}(X,p,2,1;\Delta_n)_t} \frac{\Delta^{-1/2}_n\big(\bar{V}(X,p,2,1;\Delta_n)_t-V(X,p)_t\big)}{V(X,p)_t}\bigg) \\
& \qquad + \Delta^{-1/2}_n \Bigg(\bigg(\frac{\tau_2(2\Delta_n)^2}{\tau_2(\Delta_n)^2}\bigg)^{p/2} - \big(2^{2\alpha+1}\big)^{p/2} \Bigg).
\end{aligned}
\end{multline}
When $g$ is given by \eqref{gammakernel}, one can establish the estimate
\begin{equation}\label{bias}
\Delta^{-1/2}_n\bigg(\frac{\tau_2(2\Delta_n)^2}{\tau_2(\Delta_n)^2} - 2^{2\alpha+1}\bigg) = O(\Delta^{1/2-2\alpha}_n),
\end{equation}
which is sharp in the sense the big $O$ cannot be replaced with a little $o$; we refer to section 5.2 in \cite{BCP11WP}
for the justification of \eqref{bias}.
Thus, due to the inequality $|x^r-y^r| \leq r(x^{r-1}\vee y^{r-1}) |x-y|$, where $x$, $y\geq0$ and $r\geq 1$, the latter term on right-hand side (r.h.s.)\ of \eqref{cofdecomp} is asymptotically negligible when $\alpha \in (-\frac{1}{2},0) \cup (0,\frac{1}{4})$ and $p \geq 2$. 
Applying Theorem \ref{th2}, the former term on r.h.s.\ of \eqref{cofdecomp} converges stably in law to
\[
\frac{2^{\frac{(2\alpha+1)p}{2}} (-1,1) \int_0^t |\sigma_s|^p \Lambda^{1/2}_p d B_s}{V(X,p)_t}.
\]
Invoking the delta method and the basic properties of stable convergence, we arrive at an asymptotically normal test statistic given as follows.

\begin{prop}\label{cofnorm} If the conditions of Theorem \ref{th2} hold and $g$ is given by \eqref{gammakernel} with $\alpha \in (-\frac{1}{2},0) \cup (0,\frac{1}{4})$, then for any $p \geq 2$, we have 
\[
\frac{(\hat{\alpha}(p,\Delta_n,t)-\alpha) V(X,p,2,1;\Delta_n)_t}{|h'_p(\cof(p,\Delta_n)_t)| \cof(p,\Delta_n)_t \sqrt{ m^{-1}_{2p}V(X,2p,2,1;\Delta_n) (-1,1) \Lambda_p (-1,1)^T }} \schw N(0,1).
\]
\end{prop}

\begin{rem}\label{plessthantwo}
If $p \in (\frac{1}{2},2)$, then under the stronger restriction $\alpha \in (-\frac{1}{2},\frac{p-1}{2p}) \setminus \{ 0\}$ it is possible to show that the latter term on right-hand side (r.h.s.)\ of \eqref{cofdecomp} is asymptotically negligible and, thus, the conclusion of Proposition \ref{cofnorm} holds.
\end{rem}

\begin{rem}
An alternative method to estimate the smoothness parameter is using the realized variation ratio statistic. However, it does not allow the same flexibility over the choice of the power $p$, as the change-of-frequency statistic does. We refer to \cite{BCP12} for more details on the realized variation ratio. \qed
\end{rem}

The estimation method described so far has, of course, the drawback that it applies only when $\alpha \in (-\frac{1}{2},0) \cup (0,\frac{1}{4})$. In the case $\alpha \in [\frac{1}{4},\frac{1}{2})$, the latter term on the r.h.s.\ of \eqref{cofdecomp} is non-negligible due to the sharpness of the estimate \eqref{bias}. To cover the critical region $[\frac{1}{4},\frac{1}{2})$, we develop a theory of power variations with gaps. More precisely, we consider for $v=1,2$ modified power variation statistics of the form
\begin{align*} 
V(X,p,k,u_n,v;\Delta_n)_t &:= \begin{cases}{\displaystyle\sum_{i=[k/u_n]+1}^{[t/(u_n\Delta_n)]} |\Delta_{iu_n,k}^n X|^p}, & v= 1, \\ {\displaystyle\sum_{i=[k/u_n]+1}^{[t/(u_n\Delta_n)]-1} |\Delta_{iu_n+[u_n/2],k}^{n,2} X|^p}, & v= 2, \end{cases} \\
\bar{V}(X,p,k,u_n,v;\Delta_n)_t &:= u_n \Delta_n \tau_{k}(v\Delta_n)^{-p} V(X,p,k,u_n,v;\Delta_n)_t,
\end{align*}
where $u_n \in \N$ denotes the size of the gaps. Simply put, the variation $V(X,p,k,u_n,1;\Delta_n)_t$ is computed by taking only every $u_n$-th increment into account, whereas
\begin{equation*}
V(X,p,k,u_n,2;\Delta_n)_t
\end{equation*}
is defined similarly, but choosing increments that fall between those that contribute to $V(X,p,k,u_n;\Delta_n)_t$ via the shift $[u_n/2]$ in the definition. In what follows, we let $u_n \rightarrow \infty$ so that $u_n \Delta_n \rightarrow 0$. We also assume that $n$ is always large enough so that $u_n \geq (4k+2)/3$ (to ensure that the definition above makes sense).

It is straightforward to show that a law of large numbers, a counterpart of Theorem \ref{th1}, holds also for power variation with gaps. That is, if conditions (\hyperlink{A1}{A1}) and (\hyperlink{A2}{A2}) hold, then
\begin{align} \label{llngaps}
\bar{V}(X,p,k,u_n,v;\Delta_n)_t \ucp V(X,p)_t.
\end{align}
for $v=1,2$.
Next, we state a central limit theorem for power variation with gaps, which is actually simpler than Theorem \ref{th2}. This is due to the fact that the widening gaps between the increments make them asymptotically uncorrelated.
\begin{theo} \label{th3}
Assume that the conditions (\hyperlink{A1}{A1}), (\hyperlink{A2}{A2}) hold and (\hyperlink{A3}{A3-$\gamma$}) is satisfied for some $\gamma \in (0,1)$
with $\gamma (p\wedge 1) > 1/2$. Moreover, let $u_n \rightarrow \infty$ so that $u_n \Delta_n \rightarrow 0$. If $k=1$ we further assume that $\alpha \in (-\frac 12,0)$. Then we obtain the stable convergence
\begin{multline*}
(u_n\Delta)_n^{-1/2} \Big(\bar{V}(X,p,k,u_n,1;\Delta_n)_t - V(X,p)_t,
\bar{V}(X,p,k,u_n,2;\Delta_n)_t - V(X,p)_t  \Big) \\ \stab  \sqrt{m_{2p}-m_p^2} \int_0^t |\sigma_s|^p ~dB_s
\end{multline*}
on $\mathbb D^2([0,T])$ equipped with the uniform topology, where 
$B$ is a $2$-dimensional Brownian motion that is defined on an
extension of the original probability space $(\Omega, \mathcal{F}, \mathbb P)$ and is independent of $\mathcal{F}$.
\end{theo} 
The proof of Theorem \ref{th3} is very similar to that of Theorem \ref{th1}. We provide a sketch of it, primarily to outline the key differences, in appendix \ref{proofs}.

To construct our second estimator for the smoothness parameter $\alpha$, we look into the properties of the modified change-of-frequency statistic
\[
\cof(p,\Delta_n,u_n)_t = \frac{V(X,p,2,u_n,2;\Delta_n)_t}{V(X,p,2,u_n,1;\Delta_n)_t},
\]
involving power variations with gaps. Due to the law of large numbers \eqref{llngaps}, it is immediate that
\[ 
h_p\big(\cof(p,\Delta_n,u_n)_t\big) \ucp \alpha,
\] 
under conditions (\hyperlink{A1}{A1}) and (\hyperlink{A2}{A2}), where the function $h_p$ is defined in Proposition \ref{propest}. 
The key point of using gaps is that, thanks to the slower rate of convergence $(u_n\Delta_n)^{-1/2}$ in Theorem \ref{th3}, we may write
\begin{multline}\label{cofgapsdecomp}
(u_n\Delta_n)^{-1/2} \big(\cof(p,\Delta_n,u_n)_t - 2^{\frac{(2\alpha +1)p}{2}}\big) \\
\begin{aligned}
& = \bigg(\frac{\tau_2(2\Delta_n)^2}{\tau_2(\Delta_n)^2}\bigg)^{p/2} \bigg( \frac{(u_n\Delta_n)^{-1/2}\big(\bar{V}(X,p,2,u_n,2;\Delta_n)_t-V(X,p)_t\big)}{V(X,p)_t} \\
& \qquad - \frac{\bar{V}(X,p,2,u_n,2;\Delta_n)_t}{\bar{V}(X,p,2,u_n,1;\Delta_n)_t} \frac{(u_n\Delta_n)^{-1/2}\big(\bar{V}(X,p,2,u_n,1;\Delta_n)_t-V(X,p)_t\big)}{V(X,p)_t}\bigg) \\
& \qquad + (u_n\Delta_n)^{-1/2} \Bigg(\bigg(\frac{\tau_2(2\Delta_n)^2}{\tau_2(\Delta_n)^2}\bigg)^{p/2} - \big(2^{2\alpha+1}\big)^{p/2} \Bigg),
\end{aligned}
\end{multline}
where, provided that the assumptions of Theorem \ref{th3} are met, the former term on the r.h.s.\ converges stably in law to
\[
\frac{2^{\frac{(2\alpha+1)p}{2}+1} \sqrt{m_{2p}-m_p^2} \int_0^t |\sigma_s|^p d \tilde{W}_s}{V(X,p)_t},
\]
where $\tilde{W}$ is a one-dimensional standard Brownian motion independent of $\mathcal{F}$. 
Again, let us assume that $g$ is given by \eqref{gammakernel}.
Using \eqref{bias}, we find that the latter term on the r.h.s.\ of \eqref{cofgapsdecomp} is asymptotically negligible if and only if the sizes of the gaps satisfy
\begin{equation}\label{biasgaps}
u^{-1}_n = o(\Delta_n^{4\alpha-1}).
\end{equation}
(When $\alpha \in [\frac{1}{4},\frac{1}{2})$, for instance $u_n = [\Delta^{-\kappa}_n]$, where $\kappa \in (4\alpha-1,1)\neq \emptyset$, satisfy both $u_n\Delta_n \rightarrow 0$ and \eqref{biasgaps}.) Then, we obtain the following result that characterizes an asymptotically normal test statistic.

\begin{prop} \label{proplast} If the conditions of Theorem \ref{th3} hold, $g$ is given by \eqref{gammakernel} with $\alpha \in (-\frac{1}{2},0) \cup (0,\frac{1}{2})$, and \eqref{biasgaps} holds, then for any $p\geq 2$, we have
\[
\frac{\big(h_p(\cof(p,\Delta_n,u_n)_t)-\alpha\big) V(X,p,2,1;\Delta_n)_t}{|h'_p(\cof(p,\Delta_n)_t)| \cof(p,\Delta_n)_t \sqrt{ 4 (1-m^2_p/m_{2p})V(X,2p,2,1;\Delta_n) }} \schw N(0,1).
\]
\end{prop}

\begin{rem}
Obviously, the price of using gaps is that we lose lots of observations. However, such a procedure may be optimal if the observations are highly correlated. In fact,
Lang and Roueff \cite{LR} consider the change-of-frequency statistic in the context of Gaussian stationary increment processes and show (see \cite[Theorems 1-3]{LR}) that the estimator with gaps achieves the minimax bound, which is given by $\Delta_n^{2\alpha -1}$
when $\alpha \in [\frac14, \frac 12)$. In the setting of $\mathcal BSS$ processes the estimator $\cof(p,\Delta_n,u_n)_t$
attains the minimax bound if we choose $u_n\propto \Delta_n^{1-4\alpha}$; in that case Proposition \ref{proplast} 
would contain an asymptotic bias.  \qed
\end{rem}

\section{Application to turbulence data} \label{turbulence}
\setcounter{equation}{0}

To illustrate the estimation of the smoothness parameter $\alpha$, we will consider a
data set consisting of 20 million one-point measurements of the
longitudinal component of the wind velocity in the atmospheric
boundary layer, $35$ m above ground. The measurements were
performed using a hot-wire anemometer and sampled at $5$ kHz with
a resolution of $12$ bits. The time series can be assumed to be
essentially stationary, the mean is $8.3$ m/s, and the standard deviation is
$2.3$ m/s. We refer to \cite{Dh-2000} for further details on the
data set (the data set is called no.~3 therein). The observations are
standardised to have zero mean and unit variance.

\begin{figure}
  \centering
  \includegraphics{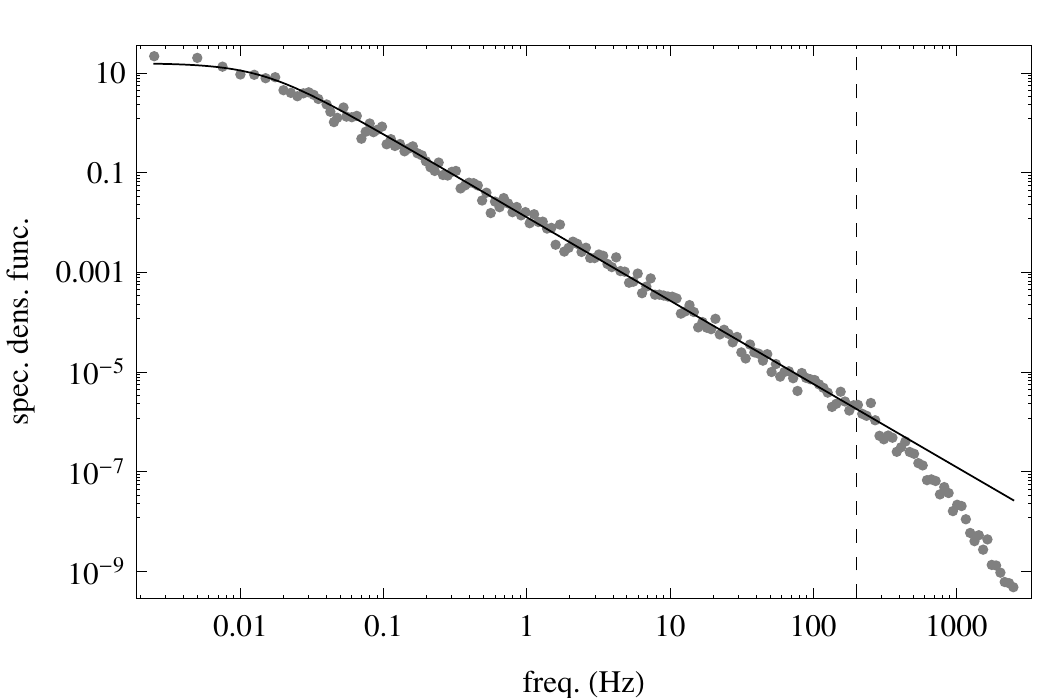}
  \caption{Gray dots: The spectral density of the measured data
    estimated using Welsh's overlapping segment averages with a
    Hanning taper, a segment length of 2~million, and an overlap of
    $50$ \%. Black curve: The spectral density~\protect\eqref{eq:1}
    with $\alpha = -0.165$ and $\lambda = 0.0884$ estimated using a
    least squares method in the double logarithmic
    representation. Only data with frequencies below $200$ Hz
    were used.}
  \label{fig:sdf}
\end{figure}

The spectral density $S$ of a $\mathcal{BSS}$ process with $g$ given by \eqref{gammakernel} satisfies
\begin{align}
  \label{eq:1}
  S(f) = \mathrm{const.}\times(1+(2\pi f/\lambda)^2)^{-(1+\alpha)},
\end{align}
where $f$ denotes the frequency. Thus, for $f\gg\lambda$, we have
$S(f) \propto f^{-2(1+\alpha)}$. In the context of turbulence,
Kolmogorov's $5/3$-law \cite{Ko-1941a} states that in the so-called
\emph{inertial range}, the spectral density is approximately proportional to $f^{-5/3}$. Putting $\alpha=-1/6$ reproduces the $5/3$-law for
$f\gg\lambda$. The weight function $g$ therefore defines an infinitely
long inertial range bounded from below, in the frequency domain,
approximately by $\lambda$.

Figure~\ref{fig:sdf} shows the spectral density estimated from the
turbulence data. For frequencies below approximately $200$ Hz, the
weight function $g$ provides a good fit with $\alpha=-0.165\approx-1/6$, in
agreement with Kolmogorov's $5/3$-law. The inertial range appears to
be from $0.1$ Hz to $200$ Hz. At higher frequencies (or,
equivalently, smaller scales) the model no longer describes the
spectral density accurately since the dissipation of the fluid's
kinetic energy into heat causes the spectral density to decay
approximately exponentially fast \cite{Si-Sm-Ya-1994}.

\begin{figure}
  \centering
  \begin{tabular}{cc}
    \includegraphics{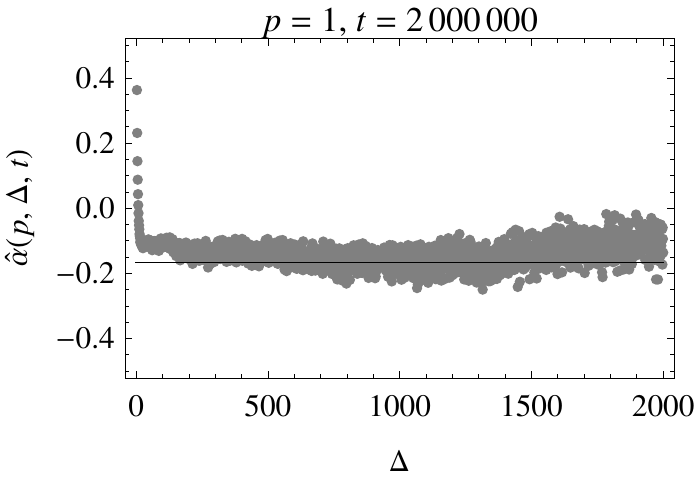} &
    \includegraphics{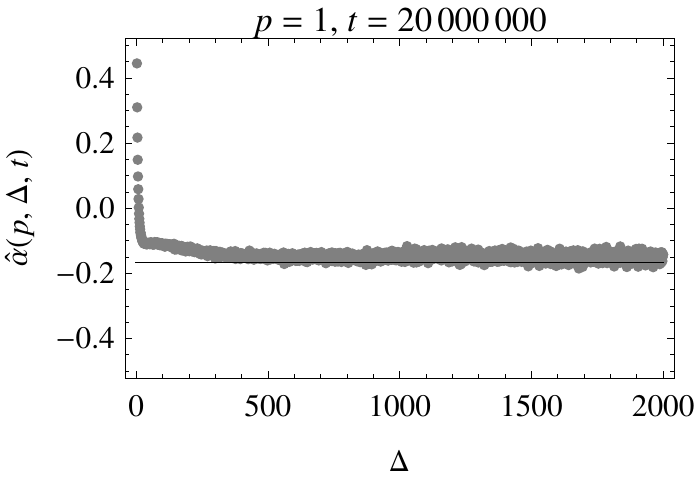} \\
    \includegraphics{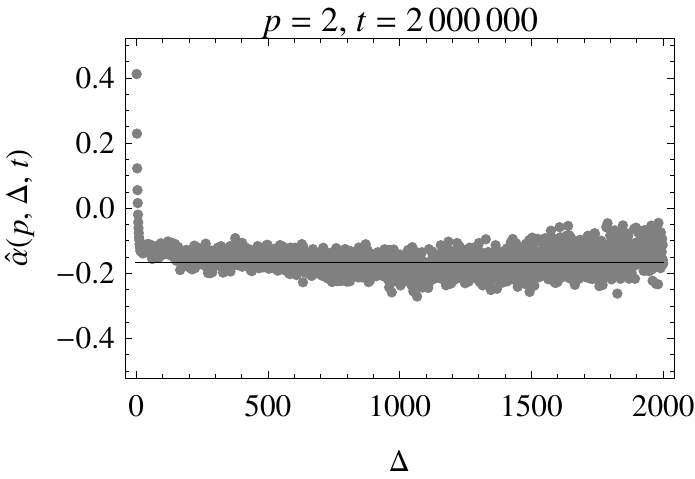} &
    \includegraphics{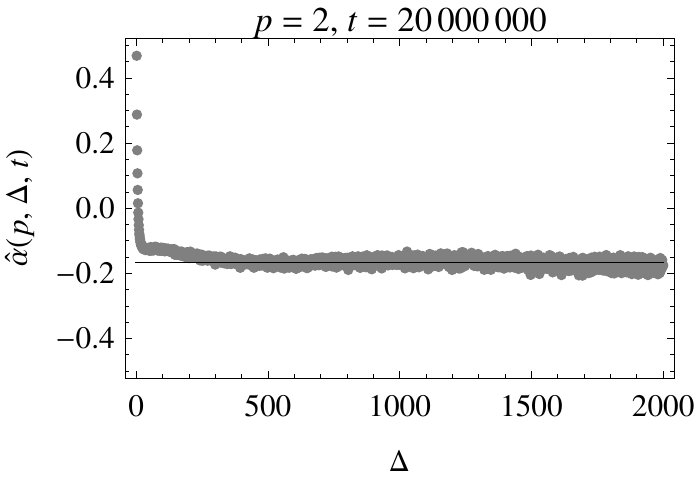} \\
    \includegraphics{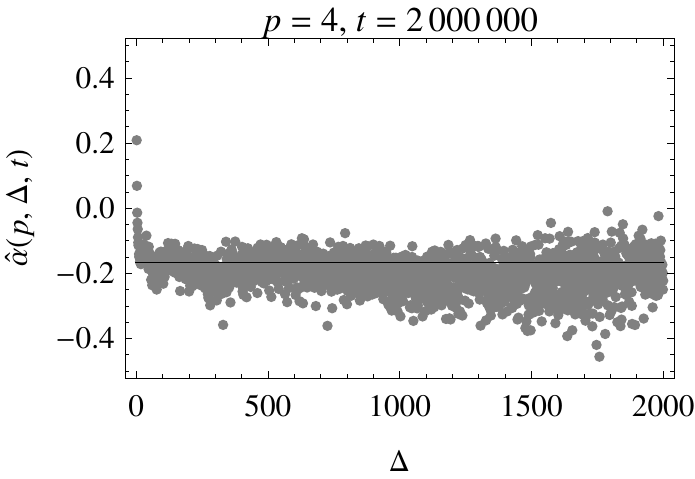} &
    \includegraphics{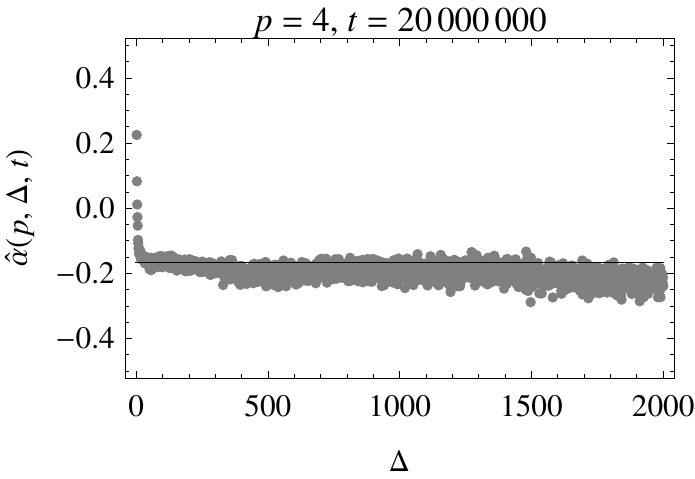} \\
    \includegraphics{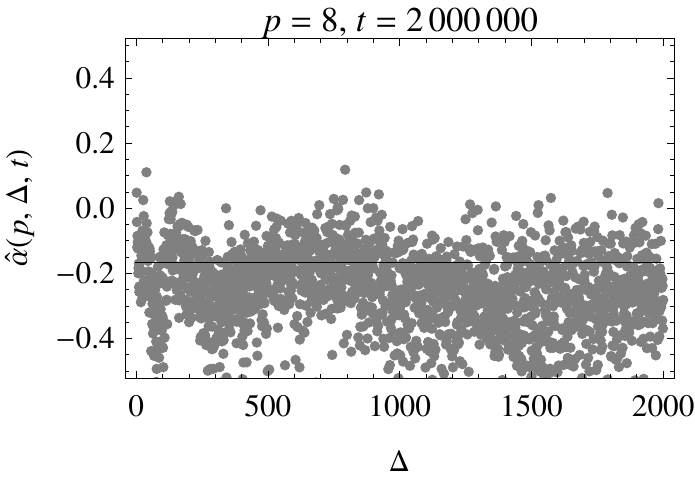} &
    \includegraphics{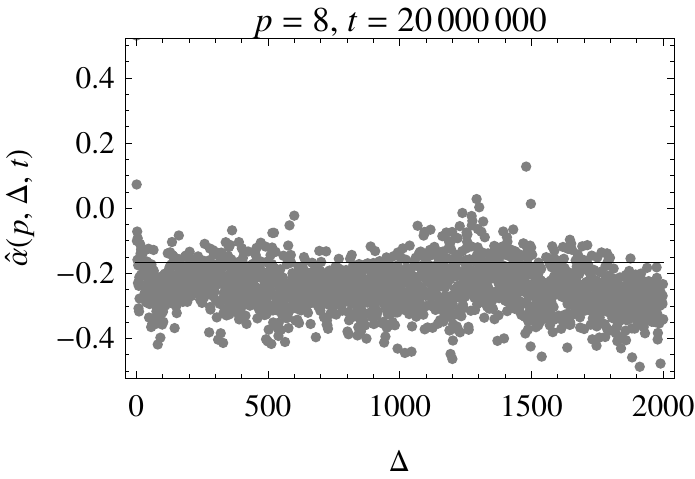}
  \end{tabular}
  \caption{Gray dots: The estimate $\hat\alpha(p,\Delta,t)$ for
    $t=2\,000\,000$ (left column) and $t=20\,000\,000$ (right column) for $p=1,2,4,8$
    (from top to bottom). The lag $\Delta$ varies from~$1$
    to~$2000$. Black line: The value $-1/6$ corresponding to
    Kolmogorov's $5/3$-law.}
  \label{fig:alpha}
\end{figure}

We use the estimator $\hat{\alpha}(p,\Delta_n,t)$ based on the change-of-frequency statistic, given by \eqref{cofconsistency}, to estimate the smoothness parameter $\alpha$. Figure~\ref{fig:alpha} shows the estimates $\hat{\alpha}(p,\Delta,t)$
for $t=2\,000\,000, 20\,000\,000$ ($\times 1/5000$ sec), $p=1,2,4,8$, and $\Delta=1,\ldots,2000$ ($\times 1/5000$ sec). Several
comments are in order. Unsurprisingly, the dispersion of the estimated
values of~$\alpha$ increases with increasing~$p$. In the case $t=20\,000\,000$, the amount of data available for the estimate of $\alpha$ is large enough to ensure a reliable estimate. We will therefore focus on the case $t=20\,000\,000$ (right
column). The best agreement with the expected $\alpha=-1/6$ is found
in the case $p=2$ when $\Delta>500$. This is perhaps not surprising,
since $p=2$ makes $\hat\alpha(p,\Delta,t)$ a second order statistic,
and the spectral density, being the Fourier transform of the
autocovariance function, is itself a second order
statistic. Furthermore, $\Delta>500$ is well within the inertial range
where the weight function $g$ has been shown to agree well with the data. 
Beyond the displayed estimates, $\hat\alpha(p,\Delta,t)$ appears fairly robust to changes in $p$ when $1.5\leq p\leq 2.5$.
At small lags, dissipation-scale effects cause the departure from
$\alpha=-1/6$ to be expected. The dependency of~$\hat\alpha(p,\Delta,t)$ on~$p$
for larger $p$ hints at some lack of robustness of the estimator. However, this is
likely due to insufficiencies of the $\mathcal{BSS}$ process in describing the data,
as simulations (Gaussian with $\sigma\equiv1$) display no such
dependency on~$p$. 

\section*{Acknowledgements}

M.S.~Pakkanen wishes to thank the Institute of Applied Mathematics at Heidelberg University for warm hospitality and acknowledge support from CREATES, funded by the Danish National Research Foundation, and from the Aarhus University Research Foundation regarding the  project ``Stochastic and Econometric Analysis of Commodity Markets".

\appendix

\section{Sketch of the proof of Theorem \ref{th3}} \label{proofs}
\setcounter{equation}{0}

The main tool for the proof of Theorem \ref{th3} relies on is the blocking technique used in \cite{BCP09,BCP11,BCP12}.
That is, for any $\ell>0$, we divide the interval $[0,t]$ into blocks of length $\ell^{-1}$. We then keep the value of the volatility process $\sigma$ fixed in each of the blocks and establish the stable convergence of the power variation. Finally, we pass to the limit $\ell \rightarrow \infty$ giving us the limit that appears in the statement of Theorem \ref{th3}.

More concretely, the blocking technique is based on the decomposition
\begin{equation*}
(u_n\Delta)_n^{-1/2} \big(\bar{V}(X,p,k,u_n,1;\Delta_n)_t - V(X,p)_t \big) = S(\ell,n)^{1}_t + S(n)^{2}_{t} + S(\ell,n)^{3}_t + S(\ell,n)^{4}_t,
\end{equation*}
where
\begin{align*}
S(\ell,n)^{1}_t & := \sum_{j=1}^{[\ell t]}|\sigma_{(j-1)/\ell}|^{p} (u_n \Delta_n)^{-1/2} \Big( u_n \Delta_n \tau_k(\Delta_n)^{-p/2}  \sum_{i \in I_\ell(j,n)} |\Delta^n_{i u_n,k} G|^p - m_p \ell^{-1} \Big), \\
S(n)^{2}_{t} &  := (u_n \Delta_n)^{1/2} \tau_k(\Delta_n)^{-p/2} \sum_{i=[k/u_n]+1}^{[t/(u_n \Delta_n)]} (|\Delta^n_{iu_n,k} X|^p- |\sigma_{(iu_n -k)\Delta_n} \Delta^n_{iu_n,k} G|^p), \\
S(\ell,n)^{3}_t &  := (u_n \Delta_n)^{1/2}\tau_k(\Delta_n)^{-p/2} \Big( \sum_{i=[k/u_n]+1}^{[t/(u_n \Delta_n)]} |\sigma_{(iu_n -k)\Delta_n} \Delta^n_{iu_n,k} G|^p
  \\
  & \qquad - \sum_{j=1}^{[\ell t]}|\sigma_{(j-1)/\ell}|^p \sum_{i \in I_\ell(j,n)} |\Delta^n_{i u_n,k} G|^p\Big), \\
S(\ell,n)^{4}_t &  := m_p (u_n \Delta_n)^{-1/2} \Big(\ell^{-1}\sum_{j=1}^{[\ell t]}|\sigma_{(j-1)/\ell}|^p - \int_0^t |\sigma_s|^p ds  \Big),
\end{align*}
with $I_\ell(j,n) := \big\{ i \in \N : i > [k/u_n], \, iu_n \Delta_n \in \big((j-1)/\ell,j/\ell \big]Ê\big\}$. An analogous decomposition can be derived also for $\bar{V}(X,p,k,u_n,2;\Delta_n)_t$ under the same centering and normalization. Using methods similar to \cite{BCP12}, we may show that for any $\delta>0$,
\begin{equation*}
\lim_{\ell \rightarrow \infty}\limsup_{n \rightarrow \infty}\PP\Big( \sup_{t \in [0,T]}|S(n)^{2}_{t} + S(\ell,n)^{3}_t + S(\ell,n)^{4}_t| > \delta \Big) = 0.
\end{equation*}
Thus, only $S(\ell,n)^{1}_t$ contributes to the limit. Let us denote by $B^1$ and $B^2$ the components of the $2$-dimensional Brownian motion $B$ appearing in the statement of Theorem \ref{th3}. By Lemma \ref{gaussiancore}, below, and the properties of stable convergence, we have that
\begin{equation}\label{blocking}
S(\ell,n)^{1}_t \stab \sqrt{m_{2p}-m^2_p}\sum_{j=1}^{[\ell t]}|\sigma_{(j-1)/\ell}|^{p} (B^1_{j/\ell}-B^1_{1,(j-1)/\ell}),Ê\quad n \rightarrow \infty.
\end{equation} 
 When $\ell \rightarrow \infty$, the r.h.s.\ of \eqref{blocking} converges in probability to
\begin{equation*}
\sqrt{m_{2p}-m^2_p} \int_0^t |\sigma_s|^p d B^1_s.
\end{equation*}
Again, analogous statements hold \emph{jointly} for $\bar{V}(X,p,k,u_n,2;\Delta_n)_t$, but $B^1$ replaced with $B^2$ therein.

To complete the argument, it remains to prove the following lemma concerning the power variations of the Gaussian core $G$.

\begin{lem}\label{gaussiancore}
Assume that the conditions (\hyperlink{A1}{A1}), (\hyperlink{A2}{A2}) hold. Moreover, let $u_n \rightarrow \infty$ so that $u_n \Delta_n \rightarrow 0$. If $k=1$ we further assume that $\alpha \in (-\frac 12,0)$. Then we obtain the stable convergence
\begin{equation*}
(u_n\Delta)_n^{-1/2} \Big(\bar{V}(G,p,k,u_n,1;\Delta_n)_t - m_p t,
\bar{V}(G,p,k,u_n,2;\Delta_n)_t - m_p t  \Big) \stab  \sqrt{m_{2p}-m_p^2} B_t
\end{equation*}
on $\mathbb D^2([0,T])$ equipped with the uniform topology, where 
$B$ is as in Theorem \ref{th3}.
\end{lem}

We need to introduce first some concepts that are needed in proof of Lemma \ref{gaussiancore}.
Let us denote by $\h$ the closed subspace of $L^2(\Omega,\f,\PP)$ spanned by the Gaussian random variables $G_t$, $t \geq 0$. Note that $\h$ is then a Hilbert space consisting of centered Gaussian random variables. Thus, the inclusion map $\h \hookrightarrow L^2(\Omega,\f,\PP)$ can be seen as an \emph{isonormal Gaussian process}, with respect to which we may define \emph{multiple Wiener integrals} of orders $l = 1,2,\ldots$, denoted by $I_l$. Recall that $I_l$ is a linear isometry $\h^{\odot l} \rightarrow L^2(\Omega,\f,\PP)$, where $\h^{\odot l}$ denotes the $l$-fold symmetric tensor product of $\h$.
In particular, $I_1(G_t) = G_t$ for all $t \geq 0$.

Suppose that $U \in \h$ is such that $\| U \|_{\h} = \E[U^2]^{1/2} =1$ and $f$ is defined by \eqref{f}. Then, using the Hermite expansion of $f$ \eqref{fherm} and the connection of Hermite polynomials and multiple Wiener integrals \cite[Proposition 1.1.4]{Nua}, we obtain a \emph{chaos expansion}
\begin{equation}\label{chaos} 
f(U) = \sum_{l=2}^\infty a_l I_l(U^{\otimes l}),
\end{equation}
strongly convergent in $L^2(\Omega,\f,\PP)$, where $U^{\otimes l}$ stands for the $l$-fold tensor product of $U$. 
Convenient sufficient conditions that functionals admitting chaos expansions of the form \eqref{chaos} converge to a Gaussian law are provided in \cite[Theorem 5]{BCP09}, building on the results of Nualart and Peccati \cite{NP}, and Peccati and Tudor \cite{PT}. We will rely on these conditions in the proof, below.
One of the conditions involve the so-called \emph{contractions} of elements of $\h^{\otimes l}$. To recall the definition, let $V_1$,\ $V_2 \in \h^{\otimes l}$, where $l \geq 2$. Then for $q = 0,1,\ldots,l-1$, the $q$-th contraction of $V_1$ and $V_2$, which is an element of $\h^{2(l-q)}$, is given by
\begin{multline*}
(V_1 \otimes_q V_2) (\omega'_1,\ldots, \omega'_{2(l-q)}) \\ := \int_{\Omega^q}V_1(\omega'_1,\ldots,\omega'_{l-q},\omega_{1},\ldots,\omega_q)V_2(\omega_{1},\ldots,\omega_q,\omega'_{l-q+1},\ldots,\omega'_{2(l-q)})\PP^{\otimes q}(d\omega_1,\ldots,d\omega_q).
\end{multline*}
In the case $q=0$ the definition reduces to the tensor product $V_1 \otimes V_2$.

\begin{proof}[Proof of Lemma \ref{gaussiancore}]
To simplify notation, let us denote for $v=1,2$,
\begin{equation*}
Z^v_t  := (u_n\Delta)_n^{-1/2} \big(\bar{V}(G,p,k,u_n,v;\Delta_n)_t - m_p t \big).
\end{equation*}
Using \eqref{chaos} and the linearity of multiple Wiener integrals, we obtain
\begin{equation*}
\begin{split}
Z^1_t & = (u_n\Delta_n)^{1/2} \sum_{i=[k/u_n]+1}^{[t/(u_n\Delta_n)]} f\bigg(\frac{\Delta_{iu_n,k}^n X}{\tau_{k}(\Delta_n)}\bigg) + O\big((u_n\Delta_n)^{1/2}\big) \\
& = \sum_{l=2}^\infty I_l \Bigg(\underbrace{a_l(u_n\Delta_n)^{1/2}\sum_{i=[k/u_n]+1}^{[t/(u_n\Delta_n)]} \bigg(\frac{\Delta_{iu_n,k}^n X}{\tau_{k}(\Delta_n)} \bigg)^{\otimes l}}_{=: F^{1,l,n}_t \in \h^{\odot l}} \Bigg)+ O\big((u_n\Delta_n)^{1/2}\big),
\end{split}
\end{equation*}
where the remainder term is uniform in $t$. Analogously,
\begin{equation*}
Z^2_t = \sum_{l=2}^\infty I_l \Bigg(\underbrace{a_l(u_n\Delta_n)^{1/2}\sum_{i=[k/u_n]+1}^{[t/(u_n\Delta_n)]-1} \bigg(\frac{\Delta_{iu_n + [u_n/2],k}^{n,2} X}{\tau_{k}(2\Delta_n)} \bigg)^{\otimes l}}_{=: F^{2,l,n}_t\in \h^{\odot l}} \Bigg)+ O\big((u_n\Delta_n)^{1/2}\big).
\end{equation*}

To verify the conditions of Theorem 5 of \cite{BCP09}, we need to estimate, for any $s \geq t$, the inner product
\begin{equation*}
\big\langle F^{1,l,n}_s, F^{1,l,n}_t\big\rangle_{\h^{\otimes l}} = a^2_l \big(t + A^1_{l,n}+A^2_{l,n} + O(u_n\Delta_n)\big),
\end{equation*}
where the remainder term is uniform in $l$ and
\begin{align*}
A^1_{l,n} &= 2 u_n \Delta_n \sum_{j=1}^{[t/(u_n \Delta_n)]-[k/u_n]-1} ([t/(u_n \Delta_n)]-[k/u_n]-j) r_{k,n}(u_n j)^l,\\
A^2_{l,n} &= u_n \Delta_n \sum_{j=1}^{[t/(u_n \Delta_n)]-[k/u_n]} c_{j,n} r_{k,n}(u_n j)^l,
\end{align*}
with $\sup_j |c_{j,n}| \leq [s/(u_n \Delta_n)] -[t/(u_n \Delta_n)]$.
By Assumptions (\hyperlink{A1}{A1}) and (\hyperlink{A2}{A2}), there exists (cf.\ Lemma 1 of \cite{BCP09}) a sequence $\big(\bar{r}_k(j)\big)_{j \geq 0}$  such that $|r_{k,n}(j)| \leq \bar{r}_k(j)$ for all $j \geq 0$ and $n \in \N$, with
\begin{equation*}
\sum_{j=0}^\infty \bar{r}_k(j)^2 < \infty.
\end{equation*}
Since $l \geq 2$, we have
\begin{align*}
|A^1_{l,n} | & \leq C\sum_{j=1}^{[t/(u_n \Delta_n)]-[k/u_n]-1} r_{k,n}(u_n j)^2 \leq C \sum_{j=1}^\infty \bar{r}_k(u_n j)^2 \leq C \sum_{j=u_n}^\infty \bar{r}_k(j)^2,\\
|A^2_{l,n} | & \leq \big(s-t + O(u_n \Delta_n) \big) \sum_{j=1}^\infty \bar{r}_k(u_n j)^2\leq C \sum_{j=u_n}^\infty \bar{r}_k(j)^2,
\end{align*}
where $C>0$ is a constant that depends on $s$, but not on $n$ nor $l$. Letting $n \rightarrow \infty$ we obtain now
\begin{equation*}
l! \big\langle F^{1,l,n}_s, F^{1,l,n}_t\big\rangle_{\h^{\otimes l}} \rightarrow l! a^2_l t
\end{equation*}
due to the assumption $u_n \rightarrow \infty$.
In the view of \eqref{variance}, it is then immediate that
\begin{equation*}
\lim_{m\rightarrow \infty} \limsup_{n \rightarrow \infty} \sum_{l=m}^\infty l! \|F^{1,l,n}_t \|^2_{\h^{\otimes l}} = 0.
\end{equation*}
We may establish in a similar manner that
\begin{equation*}
l! \big\langle F^{2,l,n}_s, F^{2,l,n}_t\big\rangle_{\h^{\otimes l}} \rightarrow l! a^2_l t\quad \textrm{and} \quad \lim_{m\rightarrow \infty} \limsup_{n \rightarrow \infty} \sum_{l=m}^\infty l! \|F^{2,l,n}_t \|^2_{\h^{\otimes l}} = 0.
\end{equation*}
Next, let us define
\begin{equation*}
\rho_{k,n}(j) := \mathrm{corr}(\Delta^{n,1}_{k,k} G,\Delta^{n,2}_{k+j,k} G), \qquad j\in \Z.
\end{equation*}
Assumptions (\hyperlink{A1}{A1}) and (\hyperlink{A2}{A2}) imply (cf.\ Lemma 1 of \cite{BCP09}) that there exist a sequence $\big(\bar{\rho}_k(j)\big)_{j \in \Z}$  such that $|\rho_{k,n}(j)| \leq \bar{\rho}_k(j)$ for all $j \in \Z$ and $n \in \N$, with
\begin{equation*}
\sum_{j=-\infty}^\infty \bar{\rho}_k(j)^2 < \infty.
\end{equation*}
Now, for any $s$, $t\in [0,T]$, we have
\begin{equation*}
\begin{split}
\big| l! \big\langle F^{1,l,n}_s, F^{2,l,n}_t\big\rangle_{\h^{\otimes l}}\big| & \leq l!a^2_l u_n\Delta_n \sum_{i=[k/u_n]+1}^{[s/(u_n \Delta_n)]}\sum_{j=[k/u_n]+1}^{[t/(u_n \Delta_n)]-1} 
|\rho_{k,n}(u_n(j-i) + [u_n/2])|^l \\
& \leq C'\sum_{j=-\infty}^\infty \bar{\rho}_k(u_n j + [u_n/2])^2 \rightarrow 0
\end{split}
\end{equation*}
when $n \rightarrow \infty$, where $C'>0$ is a constant that depends on $s$ and $t$.

Additionally, for any $m = 1,\ldots,l-1$, we estimate the contraction $F^{1,l,n}_t \otimes_m F^{1,l,n}_t$ via 
\begin{multline*}
\|F^{1,l,n}_t \otimes_m F^{1,l,n}_t \|^2_{\h^{\otimes 2(l-m)}} \\ \leq (u_n\Delta_n)^2 \sum_{i_1,i_2,i_3,i_4=1}^{[t/(u_n\Delta_n)]} |r_{k,n}(u_n |i_1-i_2 |)|^m 
|r_{k,n}(u_n |i_3-i_4 |)|^m \\ \times |r_{k,n}(u_n |i_1-i_3 |)|^{l-m} |r_{k,n}(u_n |i_2-i_4 |)|^{l-m},
\end{multline*}
the r.h.s.\ of which is bounded by a positive constant times
\begin{multline}\label{contraction}
u_n\Delta_n \sum_{i_1,i_2,i_3=1}^{[t/(u_n\Delta_n)]} |r_{k,n}(u_n i_1 )|^m 
|r_{k,n}(u_n i_2)|^m |r_{k,n}(u_n |i_1-i_3| )|^{l-m} 
|r_{k,n}(u_n |i_2-i_3|)|^{l-m} \\
= u_n\Delta_n \sum_{i_2=1}^{[t/(u_n\Delta_n)]}  \bigg( \sum_{i_1=1}^{[t/(u_n\Delta_n)]}|r_{k,n}(u_n i_1 )|^m 
|r_{k,n}(u_n |i_1-i_2| )|^{l-m} \bigg)^2.
\end{multline}
By the Cauchy--Schwarz inequality, the r.h.s.\ of \eqref{contraction} is in turn bounded by
\begin{multline*}
\sum_{i_2=1}^{[t/(u_n\Delta_n)]} |r_{k,n}(u_n i_2 )|^{2m} \bigg(1+ \sum_{i_1=1}^{[t/(u_n\Delta_n)]} |r_{k,n}(u_n i_1 )|^{2(l-m)} \bigg) \\
\leq C'' \sum_{i_2=1}^{[t/(u_n\Delta_n)]} |\bar{r}_k(u_n i_2 )|^{2} \bigg(1+ \sum_{i_1=1}^{[t/(u_n\Delta_n)]} |\bar{r}_{k}(u_n i_1 )|^{2} \bigg),
\end{multline*}
where $C''>0$ is a constant that depends on $t$. Thus, we have $\|F^{1,l,n}_t \otimes_m F^{1,l,n}_t \|^2_{\h^{\otimes 2(l-m)}} \rightarrow 0$ as $n \rightarrow \infty$. The contraction $F^{2,l,n}_t \otimes_m F^{2,l,n}_t$ can be treated in an analogous manner.

The preceding steps and Theorem 5 of \cite{BCP09} allow us to conclude that the finite dimensional distributions of the process $(Z^1_t,Z^2_t)_{t \in [0,T]}$ converge weakly to those of
\begin{equation*}
\sqrt{\sum_{l=2}^\infty l! a^2_l} \, B_t = \sqrt{m_{2p}-m^2_p} \, B_t, \qquad t \in [0,T].
\end{equation*}
Stability of the convergence and tightness in $\mathbb D^2([0,T])$ can be established similarly as in the proof of Theorem 6 of \cite{BCP09}.
\end{proof}

\end{document}